\documentclass[12pt]{amsart}

\usepackage{amssymb,amscd,amsmath,amsthm}
\numberwithin{equation}{section}
\usepackage{graphicx}
\newtheorem{thm}{Theorem}[section]

\newtheorem{prop}[thm]{Proposition}
\newtheorem{defn}[thm]{Definition}
\newtheorem{rem}[thm]{Remark}
\numberwithin{equation}{section}

\def\br{{\mathbb R}}

\begin{document}

\title[On $\xi^{(s)}$-Quadratic Stochastic Operator]{
On $\xi^{(s)}$-Quadratic Stochastic Operators on two Dimensional simplex and their behavior}

\author{Farrukh Mukhamedov}
\address{Farrukh Mukhamedov\\
 Department of Computational \& Theoretical Sciences\\
Faculty of Science, International Islamic University Malaysia\\
P.O. Box, 141, 25710, Kuantan\\
Pahang, Malaysia} \email{{\tt far75m@yandex.ru} {\tt
farrukh\_m@iium.edu.my}}

\author{Mansoor Saburov}
\address{Mansoor Saburov\\
 Department of Computational \& Theoretical Sciences\\
Faculty of Science, International Islamic University Malaysia\\
P.O. Box, 141, 25710, Kuantan\\
Pahang, Malaysia} \email{{\tt msaburov@iium.edu.my, msaburov@gmail.com}}

\author{Izzat Qaralleh}
\address{Izzat Qaralleh\\
 Department of Computational \& Theoretical Sciences\\
Faculty of Science, International Islamic University Malaysia\\
P.O. Box, 141, 25710, Kuantan\\
Pahang, Malaysia} \email{{\tt izzat\_math@yahoo.com}}


\begin{abstract}
 A quadratic stochastic operator (in short QSO) is usually used to
present the time evolution of differing species in biology. Some
quadratic stochastic operators have been studied by Lotka and Volterra. The general problem in the nonlinear operator theory is to study the behavior of
operators. This problem was not fully finished even for
quadratic stochastic operators which are the simplest nonlinear operators. To study this problem, it was investigated several classes of QSO. In this paper, we study
$\xi^{(s)}$--QSO defined on $2$D simplex. We first classify $\xi^{(s)}$--QSO
into $20$ non-conjugate classes. Further, we investigate the dynamics of three
classes of such operators.
\\[2mm]
\end{abstract}

\maketitle


\section{Introduction}

The history of quadratic stochastic operators can be traced back
to Bernstein's work \cite{1}. The quadratic stochastic operator
was considered an important source of analysis for the study of
dynamical properties and modelings in various fields such as
biology \cite{1,3,4,2,Lo,6,V1}, physics \cite {8,7},
economics and mathematics \cite{4,5,6,9}.

One of such systems which
relates to the population genetics is given by a quadratic stochastic operator \cite{1}. A quadratic stochastic
operator (in short QSO) is usually used to present the time
evolution of species in biology, which arises as follows. Consider a
population consisting of $m$ species (or traits) $1,2,\cdots,m$. We denote a set of all species (traits) by $I=\{1,2,\cdots,m\}$.  Let $x^{(0)}=\left(x_1^{(0)},\cdots,x_m^{(0)}\right)$ be
a probability distribution of species at an initial state
and $P_{ij,k}$ be a probability that individuals in the $i^{th}$ and
$j^{th}$ species (traits) interbreed to produce an individual from $k^{th}$ species (trait). Then a probability distribution $x^{(1)}=\left(x_{1}^{(1)},\cdots,x_{m}^{(1)}\right)$ of the
spices (traits) in the first generation can be found as a total
probability,  i.e.,
\begin{equation*} x_k^{(1)}=\sum_{i,j=1}^m P_{ij,k} x_i^{(0)} x_j^{(0)}, \quad k=\overline{1,m}.
\end{equation*}
This means that the association $x^{(0)} \to x^{(1)}$ defines a mapping $V$
called \textit{the evolution operator}. The population evolves by starting
from an arbitrary state $x^{(0)},$ then passing to the state
$x^{(1)}=V(x^{(0)})$ (the first generation), then to the state
$x^{(2)}=V(x^{(1)})=V(V(x^{(0)}))=V^{(2)}\left(x^{(0)}\right)$ (the second generation), and so on. Therefore, the evolution states of the population system  described by the following discrete dynamical system
$$x^{(0)}, \quad x^{(1)}=V\left(x^{(0)}\right), \quad x^{(2)}=V^{(2)}\left(x^{(0)}\right), \quad x^{(3)}=V^{(3)}\left(x^{(0)}\right) \cdots$$

In other words, a QSO describes a distribution of the next
generation if the distribution of the current generation was
given.  The fascinating applications of QSO to population
genetics were given in \cite{6}.

In \cite{11}, it was given along self-contained exposition of the recent achievements and open problems in the theory of the QSO.
The main problem in the nonlinear operator theory is to study the
behavior of nonlinear operators. This problem was not fully
finished even in the class of QSO (the QSO is the simplest nonlinear
operator). The difficulty of the problem
depends on the given cubic matrix $(P_{ijk})_{i,j,k=1}^m$. An asymptotic behavior of the QSO even on the small dimensional simplex is complicated \cite{15,SM,16,9,17}. In order to solve this problem,
many researchers always introduced a certain class of QSO and
studied their behavior. For examples, Volterra-QSO \cite{12,121,13,14,9}, permutated Volterra-QSO \cite{18,19}, Quasi-Volterra-QSO \cite{24}, $\ell$-Volterra-QSO \cite{23,231},
non-Volterra-QSO \cite{15,16}, strictly non-Volterra-QSO \cite{26}, F-QSO
\cite{25}, and non Volterra
operators generated by product measure \cite{20,21,22}. However, all these
classes together would not cover a set of all QSO. Therefore, there
are many classes of QSO which were not studied yet. Recently, in the
papers \cite{27,28}, a new class of QSO was introduced. This class
was called a $\xi^{(s)}$-QSO. In this paper, we are going to continue
the study of $\xi^{(s)}$-QSO. This class of operators depends on a
partition of the coupled index set (the coupled trait set) ${\bf P}_{m}=\{(i,j): i<j\}\subset I\times I$. In case of two dimensional simplex ($m=3$), the
coupled index set (the coupled trait set) ${\bf P}_3$ has five possible partitions. The dynamics of $\xi^{(s)}$-QSO corresponding to the point partition (the maximal partition) of ${\bf P}_3$ have been
investigated in \cite{27,28}. In the
present paper, we are going to describe and classify such operators
generated by other three partitions. Further, we also investigate
the dynamics of three classes of such operators.

The paper is organized as follows: In Sec. 2, we give some
preliminary definitions. In Sec. 3, we discuss the classification of
$\xi^{(s)}$-QSO related to $|\xi|=2.$ It turns out that some obtained
operators are $\ell$-Volterra-QSO (see \cite{23,231}), permuted
$\ell$-Volterra-QSO. The dynamics of $\ell$-Volterra-QSO are not fully
studied yet. In \cite{23,231}, some particular cases have been
investigated, which do not cover our operators. Therefore, in
further sections, we study dynamics of $\ell$-Volterra-QSO and permuted
$\ell$-Volterra-QSO. In Sec. 4, we study the behavior of
$\ell$-Volterra-QSO $V_{13}$ taken from the class $K_1.$ In Sec.5, we study the behavior of a permuted $\ell$-Volterra-QSO $V_{4}$ taken from the class $K_{4}$. Note that
$V_{4}$ is a permutation of $V_{13}$. In Sec. 6,
we study the behavior of a permuted Volterra-QSO $V_{28}$ taken
from the class $K_{19}.$  In the last section, we just highlight the dynamics of  Volterra-QSO $V_{25}$ taken from the class $K_{17}$ which was already studied in \cite{12}-\cite{13}.

\section{Preliminaries}

Recall that a quadratic stochastic operator (QSO) is a mapping of
the simplex
\begin{eqnarray}\label{1}
S^{m-1}=\left\{x=(x_{1},\cdots, x_{m})\in\mathbb{R}^m:\
\ \sum_{i=1}^mx_{i}=1, \ \  x_{i}\geq0, \ \ i=\overline{1,m} \right\}
\end{eqnarray}
into itself, of the form
\begin{eqnarray}\label{2}
x'_{k}=\sum_{i,j=1}^mP_{ij,k}x_{i}x_{j}, \ \ \ k=\overline{1,m},
\end{eqnarray}
where $V(x)=x'=(x'_1,\cdots,x'_m)$ and $P_{ij,k}$ is a coefficient of heredity, which satisfies the following conditions
\begin{eqnarray}\label{3}
P_{ij,k}\geq0,\quad P_{ij,k}=P_{ji,k},\quad \sum_{k=1}^mP_{ij,k}=1.
\end{eqnarray}
Thus, each quadratic stochastic operator $V:S^{m-1}\to S^{m-1}$ can be uniquely
defined by a cubic matrix ${\mathcal{P}}=\big(P_{ijk}\big)_{i,j,k=1}^m$ with
conditions \eqref{3}.

We denote sets of fixed points and $k-$periodic points of $V:S^{m-1}\to S^{m-1}$ by $Fix(V)$ and $Per_{k}(V)$, respectively. Due to Brouwer's fixed point theorem, one always has that $Fix(V)\neq\emptyset$ for any QSO $V$. For a given point $x^{(0)}\in S^{m-1}$, a trajectory $\{x^{(n)}\}_{n=0}^\infty$ of $V:S^{m-1}\to S^{m-1}$ starting from $x^{(0)}$ is defined by $x^{(n+1)}=V(x^{(n)})$. By $\omega_{V}\left(x^{(0)}\right)$, we denote a set of omega limiting points of the trajectory $\{x^{(n)}\}_{n=0}^\infty$. Since $\{x^{(n)}\}_{n=0}^\infty\subset S^{m-1}$ and $S^{m-1}$ is compact, one has that $\omega_{V}\left(x^{(0)}\right)\neq\emptyset$. Obviously, if ${\omega_{V}}(x^{(0)})$ consists of a single point, then the trajectory converges and a limiting point is a fixed point of $V:S^{m-1}\to S^{m-1}$.

Recall that a Volterra-QSO is defined by \eqref{2}, \eqref{3} and
the additional assumption
\begin{equation}\label{Vol}
P_{ij,k}=0  \quad \text{if}\quad  k\not\in \{i,j\}. \end{equation}

The biological treatment of condition \eqref{Vol} is clear: \textit{the
offspring repeats the genotype (trait) of one of its parents.}

One can see that a Volterra-QSO has the following form:
\begin{equation}\label{Vol2}
x'_k=x_k\left(1+\sum^m_{i=1}a_{ki}x_i\right),\ \ k\in
I,
\end{equation} where
\begin{equation}\label{8}
a_{ki}=2P_{ik,k}-1 \ \ \mbox{for}\, i\neq k\,\,\mbox{and}\
\,a_{ii}=0, \ i \in I.
\end{equation}
 Moreover,
$$a_{ki}=-a_{ik}\ \ \mbox{and} \ \ |a_{ki}| \leq 1.$$

This kind of operators intensively studied in \cite{10,12,121,13,14,9}.
Note that this operator is a discretization of the Lotka--Volterra
model \cite{Lo,V1} which models an interacting competing species in
the population system. Such a model has received considerable attention in the
fields of biology, ecology, mathematics (see for example
\cite{3,4,8,V1}).

In \cite{23}, it has been introduced a notion of $\ell$-Volterra-QSO, which generalizes a notion of Volterra-QSO. Let us recall it here.

In order to introduce a new class of QSO, we need some auxiliary
notations.

We fix $\ell \in I$ and assume that elements $P_{ij,k}$ of the matrix
$(P_{ij,k})_{i,j,k=1}^m$ satisfy
\begin{equation}\label{11}
 P_{ij,k}= 0 \ \ \mbox{if} \ \ k \not\in \{i,j\}\ \ \mbox{for }\ \mbox{any}\ \ k\in \{1,\dots,\ell\},\ \ i,j \in I,
\end{equation}
\begin{equation}\label{11a}
 P_{i_0j_0,k}> 0 \ \ \text{for some} \  (i_0,j_0),\  i_0\neq
k,\ j_0\neq k, \ \  k \in
\{\ell+1,\dots,m\}.
\end{equation}

For any fixed $\ell \in I$, the QSO defined by \eqref{2}, \eqref{3},
\eqref{11} and \eqref{11a} is called {\it $\ell$-Volterra-QSO}.

\begin{rem} Here, we stress the following points:
\begin{enumerate}
\item[1.] Note that an $\ell$-Volterra-QSO is a Volterra-QSO if and only if
$\ell= m$.

\item[2.] It is known \cite{12} that there is not a periodic
trajectory for Volterra-QSO. However, there are
such trajectories for $\ell$-Volterra-QSO \cite{23}.
\end{enumerate}
\end{rem}

By following \cite{231}, take $k\in \{1,...,\ell\}$, then $P_{kk,i}=0$
for $i\ne k$ and
$$1=\sum_{i=1}^mP_{kk,i}=P_{kk,k}+\sum_{i=\ell+1}^mP_{kk,i}.$$
By using $P_{ij,k}=P_{ji,k}$ and denoting $a_{ki}=2P_{ik,k}-1,
k\ne i$, $a_{kk}=P_{kk,k}-1$ one then gets
\begin{equation}\label{l-V}
V:
\begin{cases}
x_k'=x_k\left(1+\sum\limits^m_{i=1}a_{ki}x_i\right)\ \ \text{if} \ \  k=
\overline{1,\ell}\\
x_k' =x_k\left(1+\sum\limits^m_{i=1}a_{ki}x_i\right)+\sum\limits^m_{{i,j= 1\atop
i\neq k, j\neq k}}P_{ij,k}x_ix_j \ \ \text{if} \ \ k=\overline{\ell+1,m}.
\end{cases}
\end{equation}

This is a canonical form of $\ell$-Volterra-QSO.

Note that
$$a_{kk}\in [-1,0], \ |a_{ki}|\leq 1, \ a_{ki}+a_{ik}=2(P_{ik,i}+P_{ik,k})-2\leq 0,\ \
i,k\in I.$$

We call that an operator $V$ is {\it permuted $\ell$-Volterra-QSO},
if there is a permutation $\tau$ of the set $I$ and an
$\ell$-Volterra-QSO $V_0$ such that $(V(x))_{\tau(k)}=(V_0(x))_k$ for any
$k\in I$. In other words, $V$ can be represented as follows:

\begin{equation}\label{l-V-t}
V_\tau:
\begin{cases}
x_{\tau(k)}'
=x_k\left(1+\sum\limits^m_{i=1}a_{ki}x_i\right)\ \ \text{if} \ \ k=
\overline{1,\ell}\\
x_{\tau(k)}'=x_k\left(1+\sum\limits^m_{i=1}a_{ki}x_i\right)+\sum\limits^m_{{i,j=
1\atop i\neq k, j\neq k}}P_{ij,k}x_ix_j \ \ \text{if} \ \ k=\overline{\ell+1,m}.
\end{cases}
\end{equation}

We remark that if $\ell=m$ then a permuted $\ell$-Volterra-QSO becomes
a permuted Volterra-QSO. Some properties of such operators were
studied in \cite{13,GK}. The Dynamics of certain class of permuted
Volterra-QSO has been investigated in \cite{28}. Note that in
\cite{23,231}, it has been studied a class of
$\ell$-Volterra-QSO. An asymptotic behavior of permuted
$\ell$-Volterra-QSO has not been investigated yet. Some
particular cases has been considered in \cite{27,28}.

In this paper, we are going  to introduce a new class of QSO which
contain $\ell$-Volterra-QSO and permuted $\ell$-Volterra-QSO as a
particular case.

Note that each element $x\in S^{m-1}$ is a
probability distribution of the set $I=\{1,...,m\}.$ Let
$x=(x_{1},\cdots,x_{m})$ and $y=(y_{1},\cdots,y_{m})$ be vectors taken from $S^{m-1}$. We say that {\it $x$ is equivalent to $y$} if $x_{k}=0$
$\Leftrightarrow$ $y_{k}=0$. We denote this relation by $x\sim y$.

Let $supp(x)=\{i:x_{i}\neq0\}$ be a support of $x\in S^{m-1}$. We say that {\it $x$ is
singular to  $y$} and denote by $x\perp y$, if $supp(x)\cap supp(y)=\emptyset.$ Note that if $x,y\in S^{m-1}$
 then $x\perp y$ if and only if $(x,y)=0$, here $(\cdot,\cdot)$ stands for a standard inner product in $\br^m$.

We denote sets of coupled indexes by
$${\bf P}_m=\{(i,j):\ i<j\}\subset I\times I, \quad \Delta_m=\{(i,i): i\in I\}\subset I\times I.$$
For a given pair $(i,j)\in \mathbf{P}_m\cup \Delta_m$, we set a  vector $\mathbb{P}_{ij}=\left(P_{ij,1},\cdots, P_{ij,m}\right)$. It is clear due to the condition \eqref{3} that $\mathbb{P}_{ij}\in S^{m-1}$.

Let $\xi_1=\{A_{i}\}_{i=1}^N$ and $\xi_2=\{B_{i}\}_{i=1}^M$ be some
fixed partitions of ${\bf P}_m$ and $\Delta_m$, respectively,  i.e.
$A_{i}\bigcap A_{j}=\emptyset$, $B_{i}\bigcap B_{j}=\emptyset$, and
$\bigcup\limits_{i=1}^N A_{i}=\mathbf{P}_m$, $\bigcup\limits_{i=1}^M
B_{i}=\Delta_m$, where $N,M\leq m$.

\begin{defn}\label{Xi-QSO}
A quadratic stochastic operator $V:S^{m-1}\to S^{m-1}$ given by
\eqref{2}, \eqref{3}, is called a $\xi^{(as)}$-QSO w.r.t. the
partitions $\xi_1,\xi_2$ (where the letter "as" stands for
absolutely continuous-singular) if the following conditions are
satisfied:
 \begin{enumerate}
\item[(i)] for each $k\in\{1,\dots,N\}$ and any
$(i,j)$, $(u,v)\in A_{k}$, one has that $\mathbb{P}_{ij}\sim \mathbb{P}_{uv}$;

 \item[(ii)] for any $k\neq \ell$, $k,\ell\in\{1,\dots,N\}$ and
 any $(i,j)\in A_{k}$ and $(u,v)\in A_{\ell}$  one has  that $\mathbb{P}_{ij}\perp\mathbb{P}_{uv}$;

\item[(iii)] for each $d\in\{1,\dots,M\}$ and any
$(i,i)$, $(j,j)\in B_{d}$, one has that $\mathbb{P}_{ii}\sim
\mathbb{P}_{jj}$;

 \item[(iv)] for any $s\neq h$, $s,h\in\{1,\dots,M\}$ and
 any $(u,u)\in B_{s}$ and $(v,v)\in B_{h}$
 one has that $\mathbb{P}_{uu}\perp\mathbb{P}_{vv}$.
\end{enumerate}
\end{defn}

\begin{rem}
If $\xi_2$ is the point partition, i.e.
$\xi_2=\{\{(1,1)\},\dots\{(m,m)\}\}$, then we call the corresponding
QSO by $\xi^{(s)}$-QSO (where the letter "s" stands for
singularity), since in this case every two different vectors
$\mathbb{P}_{ii}$ and $\mathbb{P}_{jj}$ are singular. If $\xi_2$ is
the trivial, i.e. $\xi_2=\{\Delta_m\}$, then we call the
corresponding QSO by $\xi^{(a)}$-QSO (where the letter "a" stands
for absolute continuous), since in this case every two vectors
$\mathbb{P}_{ii}$ and $\mathbb{P}_{jj}$ are equivalent. We note that
some classes of $\xi^{(a)}$-QSO have been studied in \cite{MQF}. In
the present paper, we restrict ourselves to the $\xi^{(s)}$-case.
Note that, in general, the class of $\xi^{(as)}$-QSO will be studied
elsewhere in the future.
\end{rem}

\begin{rem}\label{condition(iii)} For the $\xi^{(s)}$-QSO, i.e., in the case $\xi_2=\{\{(1,1)\},\dots\{(m,m)\}\}$, the condition (iii) of Definition \ref{Xi-QSO} is trivial  and the condition (iv) means that there is a permutation $\pi$ of the set $I=\{1,\cdots, m\}$ such that $\mathbb{P}_{ii}=e_{\pi(i)}$ for any $i=\overline{1,m}$ where $e_k=(0,\cdots,0, \underbrace{1}\limits_{k}, 0,\cdots, 0),$ $k=\overline{1,m}$, are vertices of the simplex $S^{m-1}$.
\end{rem}

\textsc{A biological interpretation of a $\xi^{(s)}-$QSO:} We treat
$I=\{1,\cdots,m\}$ as a set of all possible traits of the population
system. A coefficient $P_{ij,k}$ is a probability that parents in
the $i^{th}$ and $j^{th}$ traits interbreed to produce a child from
the $k^{th}$ trait. The condition $P_{ij,k}=P_{ji,k}$ means that the
gender of parents do not influence to have a child from the $k^{th}$
trait. In this sense, $\mathbf{P}_m\cup \Delta_m$ is a set of all
possible coupled traits of parents. A vector
$\mathbb{P}_{ij}=\left(P_{ij,1},\cdots, P_{ij,m}\right)$ is a
possible distribution of children in a family while parents are
carrying traits from the $i^{th}$ and $j^{th}$ types. A biological
meaning of a $\xi^{(s)}-$QSO is as follows: a set $\mathbf{P}_m$ of
all differently coupled traits of parents is splitted  into $N$
groups $A_1,\cdots, A_N$ (here $N$ is less than the number $m$ of
traits) such that the chance (probability) of having a child from
any trait in two different family whose parents' coupled traits
belong to the same group $A_k$ is simultaneously either positive or
zero (the condition (i) of Definition \ref{Xi-QSO}), meanwhile, two
family whose parents' coupled traits belong to two different groups
$A_k$ and $A_l$ cannot have a child from the same trait,
simultaneously (the condition (ii) of Definition \ref{Xi-QSO}).
Moreover, the parents which are sharing the same type of traits can
have a child from {only} one type of traits (the condition (iv) of
Definition \ref{Xi-QSO} and Remark \ref{condition(iii)}).

 \section{Classification of $\xi^{(s)}-$QSO on 2D simplex}

In this section, we are going to study $\xi^{(s)}-$QSO in two
dimensional simplex, i.e. $m=3$. In this case, we
 have the following possible partitions of $\mathbf{P}_3$
\begin{eqnarray*}
&&\xi_{1}=\{\{(1,2)\},\{(1,3)\},\{(2,3)\}\},\ |\xi_{1}|=3,\\
&&\xi_{2}=\{\{(2,3)\},\{(1,2),(1,3)\}\},\ |\xi_{2}|=2,\\
&&\xi_{3}=\{\{(1,3)\},\{(1,2),(2,3)\}\},\ |\xi_{3}|=2,\\
&&\xi_{4}=\{\{(1,2)\},\{(1,3),(2,3)\}\}, \ |\xi_{4}|=2,\\
&&\xi_{5}=\{(1,2),(1,3),(2,3)\},\ |\xi_{5}|=1.
\end{eqnarray*}

We note that in \cite{27,28}, it has been investigated
$\xi^{(s)}$-QSO related to the partition $\xi_1$ which is the maximal partition of $\mathbf{P}_3$. In this paper, we are aiming to study $\xi^{(s)}$-QSO related to the partitions $\xi_2, \xi_3, \xi_4$. We shall show that these three classes of $\xi^{(s)}$-QSO are conjugate each other. Therefore, it is enough to study $\xi^{(s)}$-QSO related to the partition $\xi_2$. A class of $\xi^{(s)}$-QSO related to the partition $\xi_5$ will be studied in elsewhere in the future.

Let us recall that two operators $V_{1},V_{2}$ are called {\it (topologically or linearly) conjugate}, if there is a permutation matrix $P$ such that $P^{-1}
V_{1}P=V_{2}$. Let $\pi$ be a permutation of the set $I=\{1,\cdots,m\}$.  For any vector $x$, we define $\pi(x)=(x_{\pi(1)},\cdots, x_{\pi(m)})$.  It is easy to check that if $\pi$ is a permutation of the set $I$ corresponding to the given permutation matrix $P$ then one has that $Px=\pi(x)$. Therefore, two operators $V_{1},V_{2}$ are conjugate if and only if $\pi^{-1}V_1\pi=V_2$ for some permutation $\pi$. Throughout this paper, we shall consider "conjugate operators" in this sense. We say that two classes $K_1$ and $K_2$ of operators are conjugate if every operator taken from $K_1$ is conjugate to some operator taken from $K_2$ and vise versus.

\begin{prop}\label{iz0}
A class of all $\xi^{(s)}-$QSO corresponding to the partition $\xi_{3}$
(or $\xi_{4}$) is conjugate to a class of all $\xi^{(s)}-$QSO
corresponding to the partition $\xi_{2}.$
\end{prop}

\begin{proof} We show that two classes of all $\xi^{(s)}-$QSO corresponding to the partitions $\xi_{2}$ and $\xi_{3}$ are conjugate each other. Analogously, one can show that two classes of all $\xi^{(s)}-$QSO corresponding to the partitions $\xi_{2}$ and $\xi_{4}$ are conjugate each other as well.

Assume that an operator $V:S^2\to S^2$ given by
\begin{eqnarray*}
V:\ \ x'_k=\sum_{i,j=1}^3 P_{ij,k} x_ix_j, \quad k=1,2,3,
\end{eqnarray*}
is a $\xi^{(s)}-$QSO corresponding to the partition
$\xi_{3}=\{\{(1,3)\},\{(1,2),(2,3)\}\}.$ This means that the coefficients $(P_{ij,k})_{i,j,k=1}^3$ of $V$ satisfy the following three conditions:
(i) $\mathbb{P}_{12}\sim \mathbb{P}_{23}$ (ii) $\mathbb{P}_{13}\perp \mathbb{P}_{12}, \ \  \mathbb{P}_{13} \perp \mathbb{P}_{23}$ (iii) $\mathbb{P}_{11}\perp \mathbb{P}_{22}\perp \mathbb{P}_{33}$ where $\mathbb{P}_{ij}=(P_{ij,1},P_{ij,2},P_{ij,3})$.

We consider the following operator $V_{\pi}=\pi^{-1} V \pi,$ where $\pi=
\left(
\begin{array}{ccc}
1 \ \ \ 2  \ \ \ 3\\
2 \ \ \ 1 \ \ \ 3
\end{array}\right).$
It is clear that $V_{\pi}$ is conjugate to $V,$ where
\begin{eqnarray*}
V_{\pi}: \ \ x'_k=\sum_{i,j=1}^3 P_{ij,k}^{\pi} x_i x_j,\  \ k=1,2,3
\end{eqnarray*}
such that $P_{ij,k}^{\pi}=P_{\pi(i)\pi(j),\pi(k)}$ for any $i,j,k=1,2,3$, equivalently, $\mathbb{P}_{ij}^{\pi}=\pi \mathbb{P}_{\pi(i)\pi(j)}$ (in a vector
form) for any $i,j=1,2,3$. Now, we
are going to show that $V_{\pi}$ is a
$\xi^{(s)}$-QSO corresponding to
$\xi_{2}=\{\{(2,3)\},\{(1,2),(1,3)\}\}.$ In order to show it we have
to check three conditions:
\begin{enumerate}
\item[(i)] $\mathbb{P}_{12}^{\pi}\sim
\mathbb{P}_{13}^{\pi}.$ Indeed, since $\mathbb{P}_{12}^{\pi}=\pi
\mathbb{P}_{12},$ $\mathbb{P}_{13}^{\pi}=\pi \mathbb{P}_{23},$ $\mathbb{P}_{12}\sim \mathbb{P}_{23},$ one has
$\mathbb{P}_{12}^{\pi}\sim \mathbb{P}_{13}^{\pi}.$

\item[(ii)] $\mathbb{P}_{12}^{\pi}\perp
\mathbb{P}_{23}^{\pi},\mathbb{P}_{13}^{\pi}\perp
\mathbb{P}_{23}^{\pi}$. Indeed, since $\mathbb{P}_{12}^{\pi}=\pi
\mathbb{P}_{12},\ \mathbb{P}_{23}^{\pi}=\pi \mathbb{P}_{13},$ and
$\mathbb{P}_{12}\perp \mathbb{P}_{13},$ we obtain that
$\mathbb{P}_{12}^{\pi}\perp \mathbb{P}_{23}^{\pi}.$ In the same
manner, we can get that $\mathbb{P}_{13}^{\pi}\perp
\mathbb{P}_{23}^{\pi}.$

\item[(iii)] $\mathbb{P}_{11}^{\pi}\perp \mathbb{P}_{22}^{\pi}\perp
\mathbb{P}_{33}^{\pi}.$ Indeed, since $\mathbb{P}_{11}^{\pi}=\pi
\mathbb{P}_{22}$, $\mathbb{P}_{22}^{\pi}=\pi \mathbb{P}_{11},$ $\mathbb{P}_{33}^{\pi}=\pi \mathbb{P}_{33}$ and
$\mathbb{P}_{11}\perp \mathbb{P}_{22}\perp \mathbb{P}_{33}$, we have that
$\mathbb{P}_{11}^{\pi}\perp \mathbb{P}_{22}^{\pi}\perp
\mathbb{P}_{33}^{\pi}.$
\end{enumerate}

This shows that any $\xi^{(s)}-$QSO taken from the class corresponding to the partition  $\xi_{3}$ is conjugate to some $\xi^{(s)}-$QSO taken from the class corresponding to the partition  $\xi_{2}$. Analogously, we can show that any $\xi^{(s)}-$QSO $V$ taken from the class corresponding to the partition  $\xi_{2}$ is conjugate to a $\xi^{(s)}-$QSO $V_\pi=\pi^{-1} V \pi$ which belongs to the class corresponding to the partition  $\xi_{3}$, where $\pi$ is the same permutation as given above.
This completes the proof.
\end{proof}

Therefore, it is enough to study a class of all $\xi^{(s)}$-QSO corresponding to the partition $\xi_{2}$. Now, we shall consider some sub-class of a class of all $\xi^{(s)}$-QSO corresponding to the partition $\xi_{2}$ by choosing coefficients $(P_{ij,k})_{i,j,k=1}^3$ in special forms:

\begin{center}
\begin{tabular}{||l|c|c|c||} \hline\hline
Case & $\mathbb{P}_{12}$  & $\mathbb{P}_{13}$ & $\mathbb{P}_{23}$\\ \hline
$\mathbf{{{I}}}_{1}$ & $(a,1-a,0)$ & $(a,1-a,0)$ & $(0,0,1)$\\ \hline
$\mathbf{{{I}}}_{2}$ & $(0,a,1-a)$ & $(0,a,1-a)$ & $(1,0,0)$\\ \hline
$\mathbf{{{I}}}_{3}$ & $(a,0,1-a)$ & $(a,0,1-a)$ & $(0,1,0)$\\ \hline
$\mathbf{{{I}}}_{4}$ & $(0,0,1)$ & $(0,0,1)$ & $(a,1-a,0)$\\ \hline
$\mathbf{{{I}}}_{5}$ & $(1,0,0)$ & $(1,0,0)$ & $(0,a,1-a)$\\ \hline
$\mathbf{{{I}}}_{6}$ & $(0,1,0)$ & $(0,1,0)$ & $(a,0,1-a)$\\
\hline\hline
\end{tabular}\\[2mm]
\end{center}
where $a\in[0,1]$ and
\begin{center}
\begin{tabular}{||l|c|c|c||}
\hline\hline
Case & $\mathbb{P}_{11}$ & $\mathbb{P}_{22}$ & $\mathbb{P}_{33}$  \\ \hline
$\mathbf{II}_{1}$ & (1,0,0) & (0,1,0) & (0,0,1) \\ \hline
$\mathbf{II}_{2}$ & (0,1,0) & (1,0,0) & (0,0,1) \\ \hline
$\mathbf{II}_{3}$ & (0,0,1) & (0,1,0) & (1,0,0) \\ \hline
$\mathbf{II}_{4}$ & (1,0,0) & (0,0,1) & (0,1,0) \\ \hline
$\mathbf{II}_{5}$ & (0,0,1) & (1,0,0) & (0,1,0) \\ \hline
$\mathbf{II}_{6}$ & (0,1,0) & (0,0,1) & (1,0,0) \\
\hline\hline
\end{tabular}\\[3mm]
\end{center}

The choices of the cases $\left(\mathbf{{{I}}}_{i},\mathbf{{{II}}}_{j}\right)$, where $i,j=\overline{1,6}$, will give 36 operators from the class of $\xi^{(s)}-$QSO corresponding to the partition $\xi_{2}$.
Finally, we obtain 36 parametric operators which are defined as follows:
\small
\newpage
$$
V_{1}:\left\{
\begin{array}{l}
x_1'=x_1^2+2ax_1(1-x_1)\\
x_2'=x_2^2+2(1-a)x_1(1-x_1)\\
x_3'=x_3^2+2x_2x_3
\end{array} \right.\
V_{2}:\left\{
\begin{array}{l}
x_1'=x_2^2+2ax_1(1-x_1)\\
x_2'=x_1^2+2(1-a)x_1(1-x_1)\\
x_3'=x_3^2+2x_2x_3
\end{array} \right.
$$
\\
$$
V_{3}:\left\{
\begin{array}{l}
x_1'=x_3^2+2ax_1(1-x_1)\\
x_2'=x_2^2+2(1-a)x_1(1-x_1)\\
x_3'=x_1^2+2x_2x_3
\end{array} \right.\
V_{4}:\left\{
\begin{array}{l}
x_1'=x_1^2+2ax_1(1-x_1)\\
x_2'=x_3^2+2(1-a)x_1(1-x_1)\\
x_3'=x_2^2+2x_2x_3
\end{array} \right.\ \
$$
\\
$$
V_{5}:\left\{
\begin{array}{l}
x_1'=x_2^2+2ax_1(1-x_1)\\
x_2'=x_3^2+2(1-a)x_1(1-x_1)\\
x_3'=x_1^2+2x_2x_3
\end{array} \right.\ \
V_{6}:\left\{
\begin{array}{l}
x_1'=x_3^2+2ax_1(1-x_1)\\
x_2'=x_1^2+2(1-a)x_1(1-x_1)\\
x_3'=x_2^2+2x_2x_3
\end{array} \right.
$$
\\
$$
V_{7}:\left\{
\begin{array}{l}
x_1'=x_1^2+2x_2x_3\\
x_2'=x_2^2+2ax_1(1-x_1)\\
x_3'=x_3^2+2(1-a)x_1(1-x_1)
\end{array} \right.\ \
V_{8}:\left\{
\begin{array}{l}
x_1'=x_2^2+2x_2x_3\\
x_2'=x_1^2+2ax_1(1-x_1)\\
x_3'=x_3^2+2(1-a)x_1(1-x_1)
\end{array} \right.\ \
$$
\\
$$
V_{9}:\left\{
\begin{array}{l}
x_1'=x_3^2+2x_2x_3\\
x_2'=x_2^2+2ax_1(1-x_1)\\
x_3'=x_1^2+2(1-a)x_1(1-x_1)
\end{array} \right.\ \
V_{10}:\left\{
\begin{array}{l}
x_1'=x_1^2+2x_2x_3\\
x_2'=x_3^2+2ax_1(1-x_1)\\
x_3'=x_2^2+2(1-a)x_1(1-x_1)
\end{array} \right.\,
$$
\\
$$
V_{11}:\left\{
\begin{array}{l}
x_1'=x_2^2+2x_2x_3\\
x_2'=x_3^2+2ax_1(1-x_1)\\
x_3'=x_1^2+2(1-a)x_1(1-x_1)
\end{array} \right.\;
V_{12}:\left\{
\begin{array}{l}
x_1'=x_3^2+2x_2x_3\\
x_2'=x_1^2+2ax_1(1-x_1)\\
x_3'=x_2^2+2(1-a)x_1(1-x_1)
\end{array} \right.
$$
\\
$$
V_{13}:\left\{
\begin{array}{l}
x_1'=x_1^2+2ax_1(1-x_1)\\
x_2'=x_2^2+2x_2x_3\\
x_3'=x_3^2+2(1-a)x_1(1-x_1)
\end{array} \right.\,
V_{14}:\left\{
\begin{array}{l}
x_1'=x_2^2+2ax_1(1-x_1)\\
x_2'=x_1^2+2x_2x_3\\
x_3'=x_3^2+2(1-a)x_1(1-x_1)
\end{array} \right.\;
$$
\\
$$
V_{15}:\left\{
\begin{array}{l}
x_1'=x_3^2+2ax_1(1-x_1)\\
x_2'=x_2^2+2x_2x_3\\
x_3'=x_1^2+2(1-a)x_1(1-x_1)
\end{array} \right.\ \
V_{16}:\left\{
\begin{array}{l}
x_1'=x_1^2+2ax_1(1-x_1)\\
x_2'=x_3^2+2x_2x_3\\
x_3'=x_2^2+2(1-a)x_1(1-x_1)
\end{array} \right.\,
$$
\\
$$
V_{17}:\left\{
\begin{array}{l}
x_1'=x_2^2+2ax_1(1-x_1)\\
x_2'=x_3^2+2x_2x_3\\
x_3'=x_1^2+2(1-a)x(1-x)
\end{array} \right.\;
V_{18}:\left\{
\begin{array}{l}
x_1'=x_3^2+2ax_1(1-x_1)\\
x_2'=x_1^2+2x_2x_3\\
x_3'=x_2^2+2(1-a)x_1(1-x_1)
\end{array} \right.
$$
\\
$$
V_{19}:\left\{
\begin{array}{l}
x_1'=x_1^2+2ax_2x_3\\
x_2'=x_2^2+2(1-a)x_2x_3\\
x_3'=x_3^2+2x_1(1-x_1)
\end{array} \right.\,
V_{20}    :\left\{
\begin{array}{l}
x_1'=x_2^2+2ax_2x_3\\
x_2'=x_1^2+2(1-a)x_2x_3\\
x_3'=x_3^2+2x_1(1-x_1)
\end{array} \right.\;
$$
\\
$$
V_{21}:\left\{
\begin{array}{l}
x_1'=x_3^2+2ax_2x_3\\
x_2'=x_2^2+2(1-a)x_2x_3\\
x_3'=x_1^2+2x_1(1-x_1)
\end{array} \right.\ \
V_{22}:\left\{
\begin{array}{l}
x_1'=x_1^2+2ax_2x_3\\
x_2'=x_3^2+2(1-a)x_2x_3\\
x_3'=x_2^2+2x_1(1-x_1)
\end{array} \right.\,
$$
\\
$$
V_{23}:\left\{
\begin{array}{l}
x_1'=x_2^2+2ax_2x_3\\
x_2'=x_3^2+2(1-a)x_2x_3\\
x_3'=x_1^2+2x_1(1-x_1)
\end{array} \right.\;
V_{24}:\left\{
\begin{array}{l}
x_1'=x_3^2+2ax_2x_3\\
x_2'=x_1^2+2(1-a)x_2x_3\\
x_3'=x_2^2+2x_1(1-x_1)
\end{array} \right.
$$
\\
$$
V_{25}:\left\{
\begin{array}{l}
x_1'=x_1^2+2x_1(1-x_1)\\
x_2'=x_2^2+2ax_2x_3\\
x_3'=x_3^2+2(1-a)x_2x_3
\end{array} \right.\,
V_{26}:\left\{
\begin{array}{l}
x_1'=x_2^2+2x_1(1-x_1)\\
x_2'=x_1^2+2ax_2x_3\\
x_3'=x_3^2+2(1-a)x_2x_3
\end{array} \right.\;
$$
\\
$$
V_{27}:\left\{
\begin{array}{l}
x_1'=x_3^2+2x_1(1-x_1)\\
x_2'=x_2^2+2ax_2x_3\\
x_3'=x_1^2+2(1-a)x_2x_3
\end{array} \right.\ \
V_{28}:\left\{
\begin{array}{l}
x_1'=x_1^2+2x_1(1-x_1)\\
x_2'=x_3^2+2ax_2x_3\\
x_3'=x_2^2+2(1-a)x_2x_3
\end{array} \right.\,
$$
\\
$$
V_{29}:\left\{
\begin{array}{l}
x_1'=x_2^2+2x_1(1-x_1)\\
x_2'=x_3^2+2ax_2x_3\\
x_3'=x_1^2+2(1-a)x_2x_3
\end{array} \right.\;
V_{30}:\left\{
\begin{array}{l}
x_1'=x_3^2+2x_1(1-x_1)\\
x_2'=x_1^2+2ax_2x_3\\
x_3'=x_2^2+2(1-a)x_2x_3
\end{array} \right.
$$
\\
$$
V_{31}:\left\{
\begin{array}{l}
x_1'=x_1^2+2ax_2x_3\\
x_2'=x_2^2+2x_1(1-x_1)\\
x_3'=x_3^2+2(1-a)x_2x_3
\end{array} \right.\,
V_{32}:\left\{
\begin{array}{l}
x_1'=x_2^2+2ax_2x_3\\
x_2'=x_1^2+2x_1(1-x_1)\\
x_3'=x_3^2+2(1-a)x_2x_3
\end{array} \right.\;
$$
\\
$$
V_{33}:\left\{
\begin{array}{l}
x_1'=x_3^2+2ax_2x_3\\
x_2'=x_2^2+2x_1(1-x_1)\\
x_3'=x_1^2+2(1-a)x_2x_3
\end{array} \right.\ \
V_{34}:\left\{
\begin{array}{l}
x_1'=x_1^2+2ax_2x_3\\
x_2'=x_3^2+2x_1(1-x_1)\\
x_3'=x_2^2+2(1-a)x_2x_3
\end{array} \right.\,
$$
\\
$$
V_{35}:\left\{
\begin{array}{l}
x_1'=x_2^2+2ax_2x_3\\
x_2'=x_3^2+2x_1(1-x_1)\\
x_3'=x_1^2+2(1-a)x_2x_3
\end{array} \right.\;
V_{36}:\left\{
\begin{array}{l}
x_1'=x_3^2+2ax_2x_3\\
x_2'=x_1^2+2x_1(1-x_1)\\
x_3'=x_2^2+2(1-a)x_2x_3
\end{array} \right.
$$

\normalsize

\begin{thm}\label{22}
All 36 operators from the class of $\xi^{(s)}-$QSO corresponding to the partition $\xi_{2}$ defined as above are classified into 20 non-conjugate classes:
$$
\begin{array}{llll}
K_{1}=\{V_{1},V_{13}\}, &  K_{2}=\{V_{2},V_{15}\}, & K_{3}=\{V_{3},V_{14}\}, & K_{4}=\{V_{4},V_{16}\},\\
K_{5}=\{V_{5},V_{18}\}, & K_{6}=\{V_{6},V_{17}\}, &  K_{7}=\{V_{7}\},& K_{8}=\{V_{8},V_{9}\},\\
K_{9}=\{V_{10}\}, & K_{10}=\{V_{11}, \ \ V_{12}\}, & K_{11}=\{V_{19},V_{31}\}, & K_{12}=\{V_{20},V_{33}\},\\
K_{13}=\{V_{21},V_{32}\},  & K_{14}=\{V_{22},V_{34}\}, &K_{15}=\{V_{23},V_{36}\}, &K_{16}=\{V_{24},V_{35}\},\\
K_{17}=\{V_{25}\}, & K_{18}=\{V_{26},V_{27}\}, & K_{19}=\{V_{28}\},
&K_{20}=\{V_{29},V_{30}\}.
\end{array}
$$
\end{thm}

\begin{proof}  It is easy to check that the partition $\xi_{2}=\{\{(2,3)\},\{(1,2),(1,3)\}\}$ is invariant under only one permutation $\pi= \left(
\begin{array}{ccc}
1 \ \ \ 2  \ \ \ 3\\
1 \ \ \ 3 \ \ \ 2
\end{array}
\right)$. The proof of the theorem can be easily insured with respect to this permutation. It is straightforward.
\end{proof}

The main problem is to investigate the dynamics of these classes of operators. In
what follows, we are going to study three classes $K_1, K_4, K_{19}$. From the list, one can conclude that theses three classes of operators are either $\ell$-Volterra-QSO or permuted $\ell$-Volterra-QSO. The class $K_{17}$ was already studied in \cite{12,121}. The rest classes of operators would be studied in elsewhere in the future.

\section{Dynamics of $\xi^{(s)}$-QSO from the class $K_{1}$}

In this section we are going to study dynamic of $\xi^{(s)}$-QSO
from the class $K_{1}$.

We need some auxiliary facts about properties of the function $f_a:[0,1]\to[0,1]$ given by
\begin{equation}\label{fa}
f_a(x)=x^2+2ax(1-x),
\end{equation}
where $a\in [0,1].$ If $a=\frac{1}{2}$, then the function becomes
the identity mapping. Therefore, we shall consider only the case $a\neq
\frac{1}{2}$.

\begin{prop}\label{iz1}
Let $f_a:[0,1]\to[0,1]$ be a function given by \eqref{fa} where
$a\neq\frac{1}{2}.$ Then the following statements hold true:
\begin{enumerate}
\item[(i)] One has that $Fix(f_a)=\{0,1\};$
\item[(ii)] The function $f_a$ is increasing;
\item[(iii)] One has that $(a-\frac{1}{2})(f_a(x)-x)>0$ for any $x\in(0,1);$
\item[(iv)] One has that
$
\omega_{f_a}(x_0)=\left\{
\begin{array}{l}
\{0\} \ \ \text{if}   \ \ 0\leq a<\frac{1}{2}\\
\{1\} \ \ \text{if}   \ \ \frac{1}{2}< a\leq1
\end{array}
\right.
$ for any
$x_0\in(0,1).$
\end{enumerate}
\end{prop}

\begin{proof} Let $f_a:[0,1]\to[0,1]$ be a function given by \eqref{fa} where
$a\neq\frac{1}{2}.$
\begin{enumerate}
\item[(i)] In order to find fixed points of the function $f_a$, we should solve the following equation $x^2+2ax(1-x)=x$.
It follows from the last equation that
$(1-2a)(x^2-x)=0$.
Since $a\neq\frac{1}{2},$ we get that $x=0$ or $x=1.$ Therefore, $Fix(f_a)=\{0,1\}.$
\item[(ii)] Since $f_a'(x)=2x(1-a)+2a(1-x)\geq0,$ the function $f_a$ is increasing.
\item[(iii)] Since $f_a(x)-x=(1-2a)(x^2-x),$ we may get that $(a-\frac{1}{2})(f_a(x)-x)>0$.
\item[(iv)] Let $0\leq a<\frac{1}{2}$ and $x_0\in(0,1).$ Due to (iii), we have that $f_a(x_0)<x_0.$ Since $f_a$ is increasing, we obtain that $f_a^{(n+1)}(x_0)<f_a^{(n)}(x_0)$ for any $n\in \mathbb{N}$. This means that $\left\{f_a^{(n)}(x_0)\right\}_{n=1}^\infty$ is a bounded decreasing sequence. Consequently, it converges to some point $x^{*}$ and $x^{*}$ should be a fixed point, i.e., $x^{*}=0.$ This means that $w_{f_a}(x_0)=\{0\}.$ Similarly, one can show that if $\frac{1}{2}< a\leq1$ then
$\omega_{f_a}(x_0)=\{1\}$ for any $x_0\in(0,1).$ This completes the proof.
\end{enumerate}
\end{proof}

Now, we are going to study dynamics of a $\xi^{(s)}$-QSO $V_{13}:S^2\to S^2$ taken from $K_{1}$:
\begin{equation}\label{op1}
V_{13}:\left\{
\begin{array}{l}
x'_1=x_1^2+2x_1a(1-x_1)\\
x'_2=x_2^2+2x_2x_3\\
x'_3=x_3^2+2(1-a)x_1(1-x_1)
\end{array} \right.
\end{equation}
where $0\leq a\leq 1$. This operator is an $\ell$-Volterra-QSO and its behavior was not studied in
\cite{23,231}.

Let $e_1,e_2,e_3$ be the vertices of the simplex $S^2$ and $\Gamma_i=\left\{x\in S^{2}: x_i=0\right\}$ be an edge of the simplex $S^2$, where $i=1,2,3$. Let $S_{1\leq3}=\left\{x\in S^2: x_1\leq x_3\right\}$, $S_{1\geq3}=\left\{x\in S^2: x_1\geq x_3\right\}$, and $l_{13}=\left\{x\in S^2: x_1=x_3\right\}$.

\begin{thm}\label{22} Let $V_{13}:S^2\rightarrow S^2$ be a
 $\xi^{(s)}$-QSO given by \eqref{op1} and $x^{(0)}=\left( x_1^{(0)},x_2^{(0)},x_3^{(0)}\right) \\ \notin Fix(V_{13}) $ be an initial
point. Then the following statements hold true:
\begin{enumerate}
\item[(i)] One has that
$$
Fix(V_{13})=
\begin{cases}
\{e_1,e_2,e_3\} & \text{if} \quad a\neq\frac{1}{2}\\
\Gamma_2\bigcup l_{13} & \text {if} \quad a=\frac{1}{2}
\end{cases}
$$
\item[(ii)] If $0\leq a<\frac {1} {2}$ then
$$
\omega_{V_{13}}(x^{(0)})=
\begin{cases}
\{e_2\} & \text{if} \quad  x_2^{(0)}\neq 0 \\
\{e_3\} & \text{if} \quad  x_2^{(0)}=0
\end{cases}
$$
\item[(iii)] If $\frac {1} {2}<a\leq1$ then
$$
\omega_{V_{13}}(x^{(0)})=
\begin{cases}
\{e_1\} & \text{if} \quad x_1^{(0)}\neq 0 \\
\{e_2\} & \text{if} \quad x_1^{(0)}=0
\end{cases}
$$
\item[(iv)] If $ a=\frac{1}{2}$ then
$$
\omega_{V_{13}}(x^{(0)})=
\begin{cases}
\left\{\left(x_1^{(0)},0,1-x_1^{(0)}\right)\right\} & \text{if} \quad x_1^{(0)}>\frac{1}{2} \\
\left\{\left(x_1^{(0)},1-2x_1^{(0)},x_1^{(0)}\right)\right\} & \text{if} \quad x_1^{(0)}\leq\frac{1}{2}
\end{cases}
$$

\end{enumerate}
\end{thm}

\begin{proof} Let $V_{13}:S^2\rightarrow S^2$ be a
 $\xi^{(s)}$-QSO given by \eqref{op1}, $x^{(0)}=\left( x_1^{(0)},x_2^{(0)},x_3^{(0)}\right) \notin Fix(V_{13})$ be an initial
point, and $\{x^{(n)}\}_{n=0}^\infty$ be a trajectory of $V_{13}$ starting from the point $x^{(0)}$.
\begin{enumerate}
\item[(i)] In order to find fixed points of \eqref{op1}, we should solve the following system of equations:
\begin{equation}\label{op12}
\left\{
\begin{array}{l}
x_1$=$x_1^2+2x_1a(1-x_1)\\
x_2$=$x_2^2+2x_2x_3\\
x_3$=$x_3^2+2(1-a)x_1(1-x_1)
\end{array} \right.
\end{equation}
We shall separately consider two cases $a=\frac{1}{2}$ and $a\neq\frac{1}{2}
$.

Let $a\neq \frac{1}{2}.$ From the first equation of \eqref{op12}, we
get that $x_1=0$ and $x_1=1$ (see Proposition \ref{iz1} (i)). It follows from the second equation of \eqref{op12} that if
$x_1=0$ then $x_2=0, x_3 =1$ or $x_2=1, x_3=0$  and if
$x_1=1$ then $x_2=x_3=0$. This means that
$Fix(V_{13})=\{e_1,e_2,e_3\}$.

 Let $a=\frac{1}{2}.$ The first equation of \eqref{op12} takes the form $x_1=x_1$.
 From the second equation of \eqref{op12}, we get that $x_2(x_1-x_3)=0.$ This yields that $x_2=0$ or $x_1=x_3.$ In both cases, the third equation of \eqref{op12} holds true. Therefore, we have that $Fix(V_{13})=\Gamma_2\bigcup l_{13}$.

\item[(ii)] Let $0\leq a<\frac{1}{2}$.  It is clear that $x_1^{(n)}=f_a^{(n)}(x^{(0)}_1).$ Therefore, due to
Proposition \ref{iz1} (iv), we have that
$\lim\limits_{n\to\infty}x_1^{(n)}=0.$ Hence, $\omega_{V_{13}}(x^{(0)})\subset \Gamma_{1}=\{x\in S^{2}: x_1=0\}.$ Now, we shall separately consider two cases $x_2^{(0)}=0$ and $x_2^{(0)}\neq 0.$

Let $x_2^{(0)}=0.$ In this case $x_2^{(n)}=0$ and $\lim\limits_{n\to\infty}x_3^{(n)}=1.$
Hence,  $\omega_{V_{13}}(x^{(0)})=\{e_3\}$.

Let $x_2^{(0)}\neq 0$. We need the following result.

\textbf{Claim.}
One has that $V_{13}\left({S}_{1\leq3}\right)\subset{S}_{1\leq3}$. Moreover, for any $x^{(0)}\notin Fix(V_{13})$ there exists $n_0$ (depending on $x^{(0)}$) such that $\{x^{(n)}\}_{n=n_0}^\infty\subset {S}_{1\leq3}$.

{\textit{Proof of Claim}}. If $x_1\leq x_3,$ then
$x'_1=x_1^2+2ax_1(1-x_1)\leq x_3^2+2(1-a)x_1(1-x_1)=x'_3.$
This means that $V_{13}\left({S}_{1\leq3}\right)\subset{S}_{1\leq3}$.

Let $x^{(0)}$ $\notin$ ${S}_{1\leq3}$ and suppose that all elements of the trajectory belong to the set $S^2 \setminus {S}_{1\leq3}$, i.e., $\{x^{(n)}\}_{n=0}^\infty\subset S^2\setminus {S}_{1\leq3}$. It follows from $x_1^{(n)}\to 0$ and $x_3^{(n)}< x_1^{(n)}$ that $x_3^{(n)}\to 0.$ This with $\sum\limits_{i=1}^3 x_i^{(n)}=1$
implies that $x_2^{(n)}\to 1.$ On the other hand, we have that
$x_2^{(n+1)}=x_2^{(n)}(1-(x_1^{(n)}-x_3^{(n)}))\leq x_2^{(n)}.$
It yields that $\{x_2^{(n)}\}_{n=0}^\infty$ is decreasing, hence, it converges
to some point $x^{*}<1.$ This is a contradiction. This completes the proof of Claim.

Due to Claim, there exists $n_0$ such that $x_1^{(n)}\leq x_3^{(n)}$ for all $n\geq n_0$. Therefore, $x_2^{(n+1)}=x_2^{(n)}(1-(x_1^{(n)}-x_3^{(n)}))\geq x_2^{(n)}$ and $\{x_2^{(n)}\}_{n=n_0}^\infty$ is an increasing sequence which converges to $x_{2}^{*}.$ This yields that $x^{(n)}=(x_1^{(n)},x_2^{(n)},x_3^{(n)})$ converges to $(0,x_2^*,1-x_2^*),$ where $x_2^*>0$. We know that $(0,x_2^*,1-x_2^*)$ should be a fixed point. Consequently, $(0,x_2^*,1-x_2^*)=(0,1,0)$ and $\omega_{V_{13}}(x^{(0)})=\{e_2\}.$

\item[(iii)] Let $\frac{1}{2}<a\leq1.$
 Due to Proposition \ref{iz1} (iv), we have that  $\lim\limits_{n\to\infty}x_1^{(n)}=1,$ whenever $0<x_1^{(0)}<1.$  Therefore, if $x_1^{(0)}\neq 0$ then $\omega_{V_{13}}(x^{(0)})\subset\{e_1\}$.   Since $\omega_{V_{13}}(x^{(0)})$ is not empty, we obtain that $\omega_{V_{13}}(x^{(0)})=\{e_1\}.$ Let $x_1^{(0)}= 0,$ then $x_1^{(n)}= 0$ for all $n\in\mathbb{N}$.
Moreover, we have that $x_3^{(n+1)}=(x_3^{(n)})^2$ and
$\lim\limits_{n\to\infty}x_3^{(n)}=0.$ This means that
$\lim\limits_{n\to\infty}x_2^{(n)}=1.$
 Therefore, if $x_{1}^{(0)}=0$ then $\omega_{V_{13}}(x^{(0)})=\{e_2\}$.

\item[(iv)] Let $a=\frac{1}{2}$ and $x^{(0)}_1>\frac{1}{2}.$ Then
$x_1^{(n)}=x_1^{(0)}>\frac{1}{2}$ for any $n\in\mathbb{N}$. Since
$x_3^{(n)}=1-x_1^{(n)}-x_2^{(n)}$, one gets that $x_3^{(n)}<
\frac{1}{2}<x_1^{(n)}.$ This implies that $\{x_2^{(n)}\}_{n=0}^{\infty}$ is
decreasing, and hence it converges to $x_2^*.$ Consequently,
$(x_1^{(n)},x_2^{(n)},x_3^{(n)})\to (x_1^{(0)},x_2^{*},x_3^{*}).$ We know that $(x_1^{(0)},x_2^{*},x_3^{*})$ should be a fixed point. Since $x_1^{(0)}>\frac{1}{2}\geq x_3^{(*)}$, we find that $x_2^{*}=0$ and $x_3^{*}=1-x_1^{(0)}.$ This means that $\omega_{V_{13}}(x^{(0)})=\left\{\left(x_1^{(0)},0,1-x_1^{(0)}\right)\right\}$.

Let $x_1^{(0)}\leq \frac{1}{2}.$ Then $x_1^{(n)}=x_1^{(0)}\leq
\frac{1}{2}$ for any $n\in\mathbb{N}$. Since $x'_3-x'_1=x_3^2-x_1^2$, we have that $V\left({S}_{1\leq3}\right)\subset {S}_{1\leq3}$ and $V\left({S}_{1\geq3}\right)\subset {S}_{1\geq3}$. If $x^{(0)}=(x_1^{(0)},x_2^{(0)},x_3^{(0)})\in {S}_{1\leq3}$
then $x^{(n)}=(x_1^{(n)},x_2^{(n)},x_3^{(n)})\in {S}_{1\leq3}$. This yields that $\{x_2^{(n)}\}_{n=0}^{\infty}$ is decreasing and hence, it converges to $x_2^*.$
Therefore, $(x_1^{(n)},x_2^{(n)},x_3^{(n)})$ converges to $(x_1^{(0)},x_2^{(*)},x_3^{(*)})$. Since $(x_1^{(0)},x_2^{*},x_3^{*})\in {S}_{1\leq3}$ is a fixed point, we have that $x_2^{(*)}=1-2x_1^{(0)}, x_3^{(*)}=x_1^{(0)}.$ In the similar manner, one may have that if $x^{(0)}\in {S}_{1\geq3}$ then $\omega_{V_{13}}(x^{(0)})=\left\{\left(x_1^{(0)},1-2x_1^{(0)},x_1^{(0)}\right)\right\}.$
\end{enumerate}

This completes the proof.
\end{proof}

\section{Dynamics of $\xi^{(s)}$-QSO from the class $K_{4}$}

We are going to study dynamics of a $\xi^{(s)}$-QSO $V_4:S^2\to S^2$ taken from $K_{4}$:
\begin{equation}\label{op4}
V_{4}:\left\{
\begin{array}{l}
x_1'=x_1^2+2ax_1(1-x_1)\\
x_2'=x_3^2+2(1-a)x_1(1-x_1)\\
x_3'=x_2^2+2x_2x_3
\end{array} \right.
\end{equation}
 where $0\leq a\leq1$. One can immediately see that this operator is a permuted
$\ell$-Volterra-QSO. As we mentioned, the behavior of
such kinds of operators is not studied yet. It is worth mentioning that $V_{4}$ is a permutation of $V_{13}$.

Let
\begin{eqnarray*}
B(b)=\frac{3-2b-\sqrt{4b^2-8b+5}}{2}, &&\ \ \ b\in [0,1],\\
C_{\pm}(c)=\frac{1-2c\pm\sqrt{4c^2-8c+1}}{2}, && \ \ \ c\in \left[0,\frac{2-\sqrt{3}}{2}\right].
\end{eqnarray*}

It is clear that
$$0\leq B(b)=\frac{2(1-b)}{1+2(1-b)+\sqrt{1+4(1-b)^2}}\leq 1$$
$$0\leq \frac{1-2c\pm\sqrt{(1-2c)^2-4c}}{2}=C_{\pm}(c)\leq \frac{1-2c+\sqrt{4c^2+4c+1}}{2}=1$$
for any $b\in [0,1]$ and $c\in[0,\frac{2-\sqrt{3}}{2}]\subset [0,\frac{1}{2}]$.

 \begin{thm}\label{231} Let ${V}_{4}:S^2\rightarrow S^2$ be a
 $\xi^{(s)}$-QSO given by \eqref{op4}, $x^{(0)}=\left( x_1^{(0)},x_2^{(0)},x_3^{(0)}\right) \\ \notin Fix(V_{4})\bigcup Per_2(V_{4})$ be an initial point.  Then the following statements hold true:
\begin{enumerate}
 \item[(i)]
One has that
$$
Fix(V_{4})=
\begin{cases}
\bigg\{e_1,\big(0,\frac{3-\sqrt{5}}{2},\frac{-1+\sqrt{5}}{2}\big)\bigg\} & \text{if} \quad a\neq\frac{1}{2}\\
\bigg\{\left(b,B(b),1-B(b)\right)\bigg\}_{b\in[0,1]} & \text{if} \quad a=\frac{1}{2}\\
\end{cases}
$$

\item[(ii)]One
has that
$$
Per_2(V_{4})=
\begin{cases}
\{e_2,e_3\} & \text{if} \quad a\neq\frac{1}{2}\\
\bigg\{\left(c,C_{\pm}(c),1-C_{\pm}(c)\right)\bigg\}_{c\in[0,\frac{2-\sqrt{3}}{2})} & \text{if} \quad a=\frac{1}{2}\\
\end{cases}
$$

\item[(iii)] If $\frac{1}{2}<a\leq1$ then
$$
\omega_{V_{4}}(x^{(0)})=
\begin{cases}
\{e_1\} & \text{if}   \quad x_1^{(0)}\neq 0, \\
\{e_2,e_3\} & \text{if}   \quad x_1^{(0)}=0.\\
\end{cases}
$$
\item[(iv)]  If $a=\frac{1}{2}$ then
$$
\omega_{V_{4}}(x^{(0)})=
\begin{cases}
\bigg\{\left(x^{(0)}_1,C_{\pm}\left(x^{(0)}_1\right),1-C_{\pm}\left(x^{(0)}_1\right)\right)\bigg\} & \text{if}   \quad  x_1^{(0)}\in [0,\frac{2-\sqrt{3}}{2})\\
\bigg\{\left(x^{(0)}_1,B\left(x^{(0)}_1\right),1-B\left(x^{(0)}_1\right)\right)\bigg\} & \text{if}   \quad x_1^{(0)}\in  [\frac{2-\sqrt{3}}{2}, 1]
\end{cases}
$$
\end{enumerate}
\end{thm}

\begin{proof} Let $x^{(0)}=\left(x_1^{(0)},x_2^{(0)},x_3^{(0)}\right)\notin Fix(V_{4})\bigcup Per_2(V_{4})$ be an initial
point and $\{x^{(n)}\}_{n=0}^\infty$ be a trajectory of $V_{4}$ starting from the point $x^{(0)}$.
\begin{enumerate}
\item[(i)]  In order to find fixed points of $V_4$, we have to solve the following system:
\begin{equation}\label{s1}
\left\{
\begin{array}{l}
x_1=x_1^2+2ax_1(1-x_1)\\
x_2=x_3^2+2(1-a)x_1(1-x_1)\\
x_3=x_2^2+2x_2x_3
\end{array} \right.
\end{equation}

Let $a\neq \frac{1}{2}.$ From the first equation of the system \eqref{s1},
 one can find that $x_1=0$ or $x_1=1$ (see Proposition \ref{iz1} (i)). If $x_1=1$ then $x_2=x_3=0.$ If $x_1=0$ then the second equation of the system \eqref{s1} becomes as follows
\begin{equation*}
x_2^2-3x_2+1=0.
\end{equation*}
So, the solutions of this quadratic equation are $x_2^{\pm}=\
\frac{3\pm\sqrt{5}}{2}.$ We can verify that the only solution
$x_2=\frac{3-\sqrt{5}}{2}$  belongs to $[0,1]$. Therefore,
one has $x_3=\frac{-1+\sqrt{5}}{2}.$ Hence,
$Fix(V_4)=\{e_1,\big(0,\frac{3-\sqrt{5}}{2},\frac{-1+\sqrt{5}}{2}\big)\}$ whenever $a\neq \frac{1}{2}$.

Let $a=\frac{1}{2}$. The system \eqref{s1} then takes the
following form
 \begin{equation}\label{s2}
\begin{cases}
x_1=x_1\\
x_2=x_3^2+x_1(1-x_1)\\
x_3=x_2^2+2x_2x_3
\end{cases}
\end{equation}
So, by letting $x_1=b$ (any $b\in[0,1]$), the
second equation of the system \eqref{s2} can be written as follows
\begin{equation*}
x_2^2-(3-2b)x_2+(1-b)=0.
\end{equation*}
The solutions of the last equation are
$x_2^{\pm}=\frac{3-2b\pm\sqrt{4b^2-8b+5}}{2}$. One can check that the
only solution $x_2=\frac{3-2b-\sqrt{4b^2-8b+5}}{2}=B(b)$ belongs to
$[0,1]$. Therefore, one has that $Fix(V_4)=\{(b,B(b),1-B(b))\}_{b\in[0,1]}$ whenever $a=\frac{1}{2}.$

\item[(ii)] Let $a\neq\frac{1}{2}$.  Now,
we are going to show that the operator $V_4$ given by \eqref{op4}
does not have any order periodic points in the set $S^2\setminus \Gamma_1,$ where
$\Gamma_1=\{x\in S^{2}: x_1=0\}$. In fact, since the
function $f_a(x)=x^2+2ax(1-x)$ is increasing (due to Proposition
\ref{iz1} (ii)), the first coordinate of $V_4$ increases along the iteration
of $V_4$ in the set $S^2\setminus \Gamma_1$. This means that $V_4$ doesn't have any order periodic points in the set $S^2\setminus \Gamma_1$. Therefore, it is enough to find periodic points of $V_4$ in $\Gamma_1$. In this case, in order to find 2-periodic
points, we have to solve the following system of equations:
\begin{equation*}
\begin{cases}
x_1=0\\
x_2=(1-(1-x_2)^2)^2\\
x_3=1-(1-x_3^2)^2
\end{cases}
\end{equation*}
The solutions of the second equation of the last system are $0,1,\frac{3\pm\sqrt{5}}{2}$. Hence, we have that $ Per(V_4)=\{e_2,e_3\}$ whenever $a\neq \frac{1}{2}.$

Let $a=\frac{1}{2}.$ In order to find 2-periodic
points of $V_4$, we should solve the following system of equations
\begin{equation}\label{s4}
\begin{cases}
x_1=x_1\\
x_2=x_2^2(1-x_1+x_3)^2+x_1(1-x_1)\\
x_3=(x_3^2+x_1(1-x_1))(1-x_1+x_2(1-x_1+x_3)).
\end{cases}
\end{equation}
By letting $x_1=c$, where $c\in[0,1]$, the second equation of the system \eqref{s4} reduces to the following equation
\begin{equation*}
x_2^2(2-2c-x_2)^2+c(1-c)=x_2.
\end{equation*}
One can easily check that the solutions of the last equation which belong to $[0,1]$ are
only $B(c)$, $C_{\pm}(c)$ whenever $c\in[0,\frac{2-\sqrt{3}}{2})$ and $B(c)$ whenever $c\in[\frac{2-\sqrt{3}}{2},1]$. Consequently, we get that $Per_2(V_{4})=
\{\left(c,C_{\pm}(c),1-C_{\pm}(c)\right)\}_{c\in[0,\frac{2-\sqrt{3}}{2})}$.

\item[(iii)] Let $\frac{1}{2}<a\leq1$ and
$x_1^{(0)}=0$. In this case, the second coordinate of $V_4$ has the form $x'_2=h(x_2)$, where $h(x_2)=(1-x_2)^2.$ It is clear that the function $h$ is decreasing
on $[0,1]$. This yields that the function $h^{(2)}$ is increasing on [0,1]. As we already discussed in (i) and (ii) that $Fix(h)\cap [0,1]=\{\frac{3-\sqrt{5}}{2}\}$ and $Fix\left(h^{(2)}\right)\cap [0,1]=\{0,\frac{3-\sqrt{5}}{2},1\}$.
This means that the sets $[0,\frac{3-\sqrt{5}}{2}]$ and $[\frac{3-\sqrt{5}}{2},1]$ are invariant under the function $h^{(2)}$.  We immediately find (see the above discussion (ii)) that $h^{(2)}(x_2)>x_2$ whenever $x_2>\frac{3-\sqrt{5}}{2}$ and $h^{(2)}(x_2)<x_2$ whenever $x_2<\frac{3-\sqrt{5}}{2}$.  Consequently, one has that if $x_2^{(0)}\in[0,\frac{3-\sqrt{5}}{2})$ then $\omega_{h^{(2)}}(x_2^{(0)})=\{0\}$ and if $x_2^{(0)}\in(\frac{3-\sqrt{5}}{2},1]$ then $\omega_{h^{(2)}}(x_2^{(0)})=\{1\}$.  On the other hand, we have that
$$
V_{4}^{(n)}(x^{(0)})=
\begin{cases}
\left(0,h^{(2k)}\left(x_2^{(0)}\right), 1-h^{(2k)}\left(x_2^{(0)}\right)\right) & \text{if} \quad n=2k\\
\left(0,h^{(2k)}\left(h\left(x_2^{(0)}\right)\right), 1-h^{(2k)}\left(h\left(x_2^{(0)}\right)\right)\right) & \text{if} \quad n=2k+1
\end{cases}
$$
Therefore, we obtain that $\omega_{V_{4}}(x^{(0)})=\{e_2,e_3\}$ if $x_1^{(0)}=0$.

Let $\frac{1}{2}<a\leq 1$ and $0<x_1^{(0)}<1.$ In this case,
it is clear that $x_1^{(n)}=f_a^{(n)}(x^{(0)}_1).$ Therefore, due to Proposition \ref{iz1} (iv), we have that
$\lim\limits_{n\to\infty}x_1^{(n)}=1.$ This means that $\omega_{V_{4}}(x^{(0)})\subset\{e_1\}$.   Since $\omega_{V_{4}}(x^{(0)})\neq\emptyset$, one has that $\omega_{V_{4}}(x^{(0)})=\{e_1\}.$

\item[(iv)] Let $a=\frac{1}{2}$. Then the operator $V_4$ takes the following form:
\begin{equation}\label{s5}
\begin{cases}
x_1^{'}=x_1\\
x_2^{'}=x_3^2+x_1(1-x_1)\\
x_3^{'}=x_2(1-x_1+x_3)
\end{cases}
\end{equation}
It is clear that $L_{c}=\{x\in S^2: x_1=c\}$ is invariant under $V_4$ where $c\in [0,1].$ Therefore, we shall study the dynamics of $V_4$ over $L_c.$

Let $x^{(0)}_1=c\in [0,\frac{2-\sqrt{3}}{2})$ be a fixed number. Let us consider the function $h_c:[0,1-c]\to[0,1-c]$
$$h_c(x_2)=(1-c-x_2)^2+c(1-c).$$
One can show that the function $h_c$ is decreasing on $[0,1-c].$ This yields that the function $h_c^{(2)}$ is increasing. It follows from the discussion presented above (see (ii)) that $Fix(h_c)\cap[0,1-c]=\{B(c)\}$ and $Fix(h_c^{(2)})\cap[0,1-c]=\{B(c),C_\pm(c)\}$ where $C_{-}(c)<B(c)<C_{+}(c)$. Moreover, one has that
\begin{eqnarray*}
h_c[0,C_{-}(c)]\subset[C_{+}(c),1-c], && h_c[C_{-}(c),B(c)]\subset[B(c),C_{+}(c)],\\ h_c[B(c),C_{+}(c)]\subset[C_{-}(c),B(c)], &&  h_c[C_{+}(c),1-c]\subset[0,C_{-}(c)],
\end{eqnarray*}
Therefore, the sets $[0,C_{-}(c)]$, $[C_{-}(c),B(c)]$, $[B(c),C_{+}(c)]$, $[C_{+}(c),1-c]$ are invariant under the function $h_c^{(2)}$. Simple calculations show that
\begin{eqnarray*}
h_c^{(2)}(x_2)>x_2, &&  \forall \ x_2\in\bigg[0,C_{-}(c)\bigg)\cup \bigg(B(c),C_{+}(c)\bigg),\\
h_c^{(2)}(x_2)<x_2, &&  \forall \ x_2\in\bigg(C_{-}(c),B(c)\bigg)\cup \bigg(C_{+}(c),1-c\bigg].
\end{eqnarray*}
Consequently, we get that
\begin{eqnarray*}
\omega_{h_c^{(2)}}(x_2^{(0)})=\{C_{-}(c)\}, &&  \forall \ x_2^{(0)}\in\bigg[0,B(c)\bigg),\\
\omega_{h_c^{(2)}}(x_2^{(0)})=\{C_{+}(c)\}, &&  \forall \ x_2^{(0)}\in\bigg(B(c),1-c\bigg].
\end{eqnarray*}

On the other hand, we have that
$$
V_{4}^{(n)}(x^{(0)})=
\begin{cases}
\left(c,h_c^{(2k)}\left(x_2^{(0)}\right), 1-h_c^{(2k)}\left(x_2^{(0)}\right)\right) & \text{if} \quad n=2k\\
\left(c,h_c^{(2k)}\left(h_c\left(x_2^{(0)}\right)\right), 1-h_c^{(2k)}\left(h_c\left(x_2^{(0)}\right)\right)\right) & \text{if} \quad n=2k+1
\end{cases}
$$
Therefore, we obtain that
$$\omega_{V_{4}}(x^{(0)})=\bigg\{\bigg(c,C_{\pm}(c),1-C_{\pm}(c)\bigg)\bigg\} \quad \text{if} \quad x^{(0)}_1=c\in \left[0,\frac{2-\sqrt{3}}{2}\right).$$
In the same manner, one can show that
$\omega_{V_{4}}(x^{(0)})=\bigg\{\bigg(c,B(c),1-B(c)\bigg)\bigg\}$ whenever $x^{(0)}_1=c\in \left[\frac{2-\sqrt{3}}{2},1\right]$
\end{enumerate}
This completes the proof.
\end{proof}

\section{Dynamics of $\xi^{(s)}$-QSO from the class $K_{19}$}

We are going to study dynamics of a $\xi^{(s)}$-QSO $V_{28}:S^2\to S^2$ taken from $K_{19}$:
\begin{equation}\label{op}
V_{28}:\left\{
\begin{array}{l}
x'_1=x_1^2+2x_1(1-x_1)\\
x'_2=x_3^2+2ax_2x_3\\
x'_3=x_2^2+2(1-a)x_2x_3
\end{array} \right.
\end{equation}
 where $0\leq a\leq1$. One can see that this operator is a permuted
Volterra-QSO. The behavior of
this operator was not studied in \cite{GK,18,19}. It is worth mentioning that $V_{28}$ is a permutation of $V_{25}$.

Let $A=\frac{3-2a-\sqrt{4+(2a-1)^2}}{2(1-2a)}$ for any $a\neq \frac{1}{2}$. Then $0\leq A\leq 1$. In fact, one has that
$$0\leq A=\frac{2}{2+(1-2a)+\sqrt{4+(2a-1)^2}}\leq 1.$$

 \begin{thm}\label{23} Let ${V}_{28}:S^2\rightarrow S^2$ be a
 $\xi^{(s)}$-QSO given by \eqref{op}, $x^{(0)}=\left( x_1^{(0)},x_2^{(0)},x_3^{(0)}\right) \\ \notin Fix(V_{28})\bigcup Per_2(V_{28})$ be an initial point. Then the following statements hold true:
\begin{enumerate}
 \item[(i)] One has that
$$Fix(V_{28})=
\begin{cases}
\{e_1,(0,A,1-A)\} &  \text{if} \quad a\neq\frac{1}{2}\\
\{e_1,(0,\frac{1}{2},\frac{1}{2})\} &  \text{if} \quad a=\frac{1}{2}
\end{cases}
$$
\item[(ii)]One
has that
$$
Per_2(V_{28})=
\begin{cases}
\{e_2,e_3\} & \text{if} \quad a\neq\frac{1}{2}\\
\Gamma_1\setminus \{(0,\frac{1}{2},\frac{1}{2})\}& \text{if} \quad  a=\frac{1}{2}
\end{cases}
$$

\item[(iii)]  If $a\neq\frac{1}{2}$ then
$$
\omega_{V_{28}}(x^{(0)})=
\begin{cases}
\{e_2,e_3\}  & \text{if}   \quad x_1^{(0)}=0 \\
\{e_1\}  & \text{if}   \quad x_1^{(0)}\neq 0
\end{cases}
$$

\item[(iv)] If $a=\frac{1}{2}$ then $\omega_{V_{28}}(x^{(0)})=\{e_1\}.$
\end{enumerate}
\end{thm}

\begin{proof} Let $x^{(0)}=\left( x_1^{(0)},x_2^{(0)},x_3^{(0)}\right)\notin Fix(V_{28})\bigcup Per_2(V_{28})$ be an initial
point and $\{x^{(n)}\}_{n=0}^\infty$ be a trajectory of $V_{28}$ starting from the point $x^{(0)}$.

\begin{enumerate}
\item[(i)]  In order to find fixed points of $V_{28}$, we need to
 solve the following system of equations:
 \begin{equation}\label{op13}
\left\{
\begin{array}{l}
x_1$=$x_1^2+2x_1(1-x_1),\\
x_2$=$x_3^2+2ax_2x_3,\\
x_3$=$x_2^2+2(1-a)x_2x_3.
\end{array} \right.
\end{equation}

From the first equation of \eqref{op13} one can find that
$x_1=0$ or $x_1=1$. If $x_1=1$ then $x_2=x_3=0.$ If $x_1=0$ then
$x_2+x_3=1.$ So, the second equation of \eqref{op13} becomes as
follows:
\begin{equation}\label{a}
(1-x_2)^2+2ax_2(1-x_2)=x_2.
\end{equation}

Let $a\neq\frac{1}{2}.$ Then one can find
that solutions of \eqref{a} are
$x_2^{\pm}=\frac{3-2a\pm\sqrt{4+(2a-1)^2}}{2(1-2a)}$. We can verify that the only solution which lies in the interval
$[0,1]$ is $x_2=\frac{3-2a-\sqrt{4+(2a-1)^2}}{2(1-2a)}=A.$
Therefore, we have $x_3=1-A=\frac{-1-2a+\sqrt{4+(2a-1)^2}}{2(1-2a)}.$
Hence, $Fix({V}_{28})= \{e_{1},(0,A,1-A)\}$ whenever $a\neq \frac{1}{2}.$

Let $a=\frac{1}{2}.$ Then \eqref{a} has a solution $x_2=\frac{1}{2}.$ This yields that $x_3=\frac{1}{2}.$ Therefore, one has that $Fix({V}_{28})= \{e_{1},(0,\frac{1}{2},\frac{1}{2})\}$ whenever $a= \frac{1}{2}$.

\item[(ii)] It is clear that $V_{28}$ does not have any order periodic points in $S^{2}\setminus \Gamma_1$ (see Proposition \ref{iz1} (ii)), where $\Gamma_1=\{x\in S^2: x_1=0\}$. So, any order periodic points of $V_{28}$ lie on $\Gamma_1$ (if any). In order to find 2-periodic points of ${V}_{28}$, we have to solve the equation
 ${V}_{28}^{2}(x)=x$ with the condition $x_1=0$. Then, by taking into account $x_2+x_3=1$, we may get the following equation
\begin{equation*}
\left[x_2^2+2(1-a)x_2(1-x_2)\right]^2+2a\left[(1-x_2)^2+2ax_2(1-x_2)][x_2^2+2(1-a)x_2(1-x_2)\right]=x_2.
\end{equation*}

Let $a\neq\frac {1}{2}.$ Then, the last equation has the solutions $\{0,1,\pm A\}$. So, 2-periodic points of $V_{28}$ are only $e_2=(0,1,0)$ and $e_3=(0,0,1)$.

Let $a=\frac {1}{2}.$ Then, the equation given above becomes an identity $x_2=x_2$. This means that all points of the edge $\Gamma_1$ except $(0,\frac{1}{2},\frac{1}{2})$ are 2-periodic points.

\item[(iii)] Let $a\neq\frac {1}{2}$. It is clear that the edge $\Gamma_1$ is invariant under
${V_{28}}$. We want to study the behavior of $V_{28}$ over this
line. In this case, $V_{28}\mid_{\Gamma_1}$ takes the following
form:
$$V_{28}\mid_{\Gamma_1}:\left\{
\begin{array}{l}
x_1'=0\\
x_2'=x_3^2+2ax_2x_3\\
x_3'=x_2^2+2(1-a)x_2x_3
\end{array} \right.$$

Let us consider the function $g_a(x_2)=(1-x_2)^2+2ax_2(1-x_2)$, where $a\neq\frac {1}{2}$. One can  easily check that $g_a$ is decreasing on $[0,1]$. This yields that $g^{(2)}_a$ is increasing on $[0,1]$. As we already discussed that $Fix(g_a)\cap [0,1]=\{A\}$ and $Fix\left(g^{(2)}_a\right)\cap [0,1]=\{0,A,1\}$. This means that the sets $[0,A]$ and $[A,1]$ are invariant under the function $g_a^{(2)}$. We immediately find (see the above discussion (ii)) that $g^{(2)}_a(x_2)>x_2$ whenever $x_2>A$ and $g^{(2)}_a(x_2)<x_2$ whenever $x_2<A$.  Consequently, one has that if $x_2^{(0)}\in[0,A)$ then $\omega_{g_a^{(2)}}(x_2^{(0)})=\{0\}$ and if $x_2^{(0)}\in(A,1]$ then $\omega_{g_a^{(2)}}(x_2^{(0)})=\{1\}$. On the other hand, we have that
$$
V_{28}^{(n)}(x^{(0)})=
\begin{cases}
\left(0,g_a^{(2k)}\left(x_2^{(0)}\right), 1-g_a^{(2k)}\left(x_2^{(0)}\right)\right) & \text{if} \quad n=2k\\
\left(0,g_a^{(2k)}\left(g_a\left(x_2^{(0)}\right)\right), 1-g_a^{(2k)}\left(g_a\left(x_2^{(0)}\right)\right)\right) & \text{if} \quad n=2k+1
\end{cases}
$$
Therefore, we obtain that $\omega_{V_{28}}(x^{(0)})=\{e_2,e_3\}$ if $x_1^{(0)}=0$.

Let $x_1^{(0)}\neq 0$. It is clear that $x_1^{(n)}=f_1^{(n)}(x^{(0)}_1).$ Therefore, due to Proposition \ref{iz1} (iv), we have that
$\lim\limits_{n\to\infty}x_1^{(n)}=1.$ This means that $\omega_{V_{28}}(x^{(0)})\subset\{e_1\}$.   Since $\omega_{V_{28}}(x^{(0)})$ is not empty, we obtain that $\omega_{V_{28}}(x^{(0)})=\{e_1\}.$

\item[(iv)] Let $a=\frac{1}{2}$. Since $x^{(0)}\notin Fix(V_{28})\bigcup Per_2(V_{28}),$ we have that $x_1^{(0)}>0$. Then, due to Proposition \ref{iz1} (iv), we again has that $x_1^{(n)}=f_1^{(n)}(x^{(0)}_1)$ and $\lim\limits_{n\to\infty}x_1^{(n)}=1$.  Since $\omega_{V_{28}}(x^{(0)})$ is not empty, we obtain that $\omega_{V_{28}}(x^{(0)})=\{e_1\}.$
\end{enumerate}

This completes the proof.
\end{proof}

\section{Dynamics of $\xi^{(s)}$-QSO from the class $K_{17}$}

We are going to highlight the dynamics of a $\xi^{(s)}$-QSO $V_{25}:S^2\to S^2$
taken $K_{17}$
\begin{equation}\label{op2}
V_{25}:\left\{
\begin{array}{l}
x_1'=x_1^2+2x_1(1-x_1)\\
x_2'=x_2^2+2ax_2x_3\\
x_3'=x_3^2+2(1-a)x_2x_3
\end{array} \right.
\end{equation}
where $0\leq a\leq1.$ One can immediately see that the operator \eqref{op2} is a Volterra-QSO. The dynamics of such kinds of operators have been studied in \cite{12,121,13}. By means of the results of the mentioned papers, one can formulate the following

\begin{thm} Let ${V}_{25}:S^2\rightarrow S^2$ be a
 $\xi^{(s)}$-QSO given by \eqref{op2} and $x^{(0)}=\left( x_1^{(0)},x_2^{(0)},x_3^{(0)}\right) \\ \notin Fix(V_{25})$ be an initial point. Then the following statements hold true:
\begin{enumerate}
\item[(i)] One has that
$$Fix(V_{25})=
\begin{cases}
\{e_1,e_2,e_3\} & \text{if} \quad a\neq\frac{1}{2}\\
\Gamma_1 & \text{if} \quad a=\frac{1}{2}
\end{cases}
$$

\item[(ii)] If $0\leq a<\frac {1}{2}$ then
$$\omega_{V_{25}}(x^{(0)})=
\begin{cases}
\{e_3\} & \text{if} \quad x_1^{(0)}=0\\
\{e_1\} & \text{if} \quad x_1^{(0)}\neq0
\end{cases}
$$

\item[(iii)] If $\frac {1}{2}<a\leq1$ then
$$\omega_{V_{25}}(x^{(0)})=
\begin{cases}
\{e_2\} & \text{if} \quad x_1^{(0)}=0\\
\{e_1\} & \text{if} \quad x_1^{(0)}\neq0
\end{cases}
$$

\item[(iv)] If $a=\frac {1}{2}$ then $\omega_{V_{25}}(x^{(0)})=\{e_1\}.$
\end{enumerate}
\end{thm}

\subsection*{Acknowledgments}
The authors acknowledge the MOHE grant ERGS13-024-0057  and the IIUM
grant EDW B 13-019-0904 for the financial support.

\end{document}